\documentclass[12pt]{amsart}

\usepackage{amsfonts,amsmath,latexsym,amssymb,verbatim,amsbsy,times}
\usepackage{amsthm}
\usepackage{pstricks}
\usepackage{graphicx}
\usepackage{caption}
\usepackage[top=0.8 in, bottom=0.8in, left=0.8in, right=0.8in]{geometry}

\usepackage[colorlinks=true, pdfstartview=FitV, linkcolor=blue,citecolor=green, urlcolor=blue]{hyperref}

\usepackage{enumitem}

\theoremstyle{plain}
\newtheorem{THEOREM}{Theorem}[section]

\newtheorem{PROP}[THEOREM]{Proposition}

\theoremstyle{definition}

\theoremstyle{remark}
\newtheorem{REMARK}[THEOREM]{Remark}

\newtheorem{EXAMPLE}[THEOREM]{Example}

\newcommand{\N}{\ensuremath{\mathbb{N}}}   
\newcommand{\R}{\ensuremath{\mathbb{R}}}   
\newcommand{\T}{\ensuremath{\mathbb{T}}}   

\def \a {\alpha}
\def \b {\beta}
\def \d {\delta}
\def \g {\gamma}
\def \e {\varepsilon}
\def \f {\varphi}

\def \k {\kappa}
\def \l {\lambda}
\def \L {\Lambda}

\def \n {\nabla}

\def \O {\Omega}

\def\cprime{$'$}

\def \loc {\mathrm{loc}}

\def \< {\langle}
\def \> {\rangle}
\def \p {\partial}

\DeclareMathOperator{\diam}{diam} %
\DeclareMathOperator{\diver}{div} %
\DeclareMathOperator{\tr}{Tr} %

\newcommand{\dist}[2]{\mathrm{dist}(#1,#2)}

\def \dd  {\mathrm{d}}
\def \dt  {\, \mathrm{d}t}

\def \dx  {\, \mathrm{d}x}
\def \dy  {\, \mathrm{d}y}

\def \dm  {\, \mathrm{d}m}
\def \dr  {\, \mathrm{d}r}

\def \ds  {\, \mbox{d}s}

\begin{document}

\title[Lagrangian Trajectories for the One-Dimensional Euler Alignment Model]{On the  Lagrangian Trajectories for the One-Dimensional Euler Alignment Model without Vacuum Velocity}

\author{Trevor M. Leslie}
\address{Department of Mathematics, University of Wisconsin, Madison}
\email{tleslie2@wisc.edu}

\date{\today}



\thanks{\textbf{Acknowledgment.} Part of this work was completed while the author was supported by NSF grant DMS-1147523 (PI: Andreas Seeger). The author thanks Javier Morales, Roman Shvydkoy, Eitan Tadmor, Changhui Tan, and Yao Yao for productive discussions.}

\begin{abstract}
A well-known result of Carrillo, Choi, Tadmor, and Tan \cite{CCTT2016} states that the 1D Euler Alignment model with smooth interaction kernels possesses a `critical threshold' criterion for the global existence or finite-time blowup of solutions, depending on the global nonnegativity (or lack thereof) of the quantity $e_0 = \p_x u_0 + \phi*\rho_0$.  In this note, we rewrite the 1D Euler Alignment model as a first-order system for the particle trajectories in terms of a certain primitive $\psi_0$ of $e_0$; using the resulting structure, we give a complete characterization of global-in-time existence versus finite-time blowup of regular solutions that \emph{does not require a velocity to be defined in the vacuum}.  We also prove certain upper and lower bounds on the separation of particle trajectories, valid for smooth and weakly singular kernels, and we use them to weaken the hypotheses of Tan \cite{Tan2019WeakSing} sufficient for the global-in-time existence of a solution in the weakly singular case, when the order of the singularity lies in the range $s\in (0,\frac12)$.
\end{abstract}	

\maketitle

\section{Introduction}

\subsection{The Euler Alignment and Cucker--Smale Models}

Consider the Euler Alignment model on $\R^n$:
\begin{equation}
\label{e:m}
\left\{ 
\begin{split}
& \p_t\rho(x,t) + \diver_x(\rho u)(x,t) = 0, \quad \quad x\in \O(t):=\{x\in \R^n:\rho(t)>0\}, \\
& \p_t u(x,t) + u\cdot \n u(x,t) = \k\int_{\O(t)} \phi(x-y)(u(y,t) - u(x,t))\rho(y,t)\dy, \\
& u(x,0) = u_0(x); \; \rho(x,0) = \rho_0(x)\ge 0; \; \O(t) = \O.
\end{split}\right. 
\end{equation}

Here $\rho$ denotes the density profile, which is positive inside the time-dependent set $\O(t)$ and zero elsewhere.  The velocity $u$ is defined inside $\O(t)$, and $\O(t)$ evolves according to the flow $X:\O\to \O(t)$ generated by $u$.  That is, $\dot{X}(\a,t) = u(X(\a, t),t)$, $X(\a,0) = \a$, and $\O(t) = X(\O,t)$, where $\O = \O(0)$ is a given open, bounded subset of $\R^n$. The function $\phi$ represents the (nonnegative) communication protocol, and the parameter $\k>0$ governs the strength of the communications.  The set $\O$ is usually taken to be bounded; boundedness of $\O(t)$ is then propagated by the flow map. We make the global assumption that $\phi$ is radial (and usually radially decreasing).  

The system \eqref{e:m} constitutes a hydrodynamic description of the celebrated Cucker--Smale model of ODEs \cite{CS2007a}, \cite{CS2007b}, which has received a great deal of attention in recent years:
\begin{equation}
\label{e:CS}
\left\{ 
\begin{split}
& \dot{x}_i = v_i, \\
& \dot{v}_i = \k \sum_{j=1}^N \phi(x_{ij}) v_{ji} m_j, \\
& x_i(0) = x_{i0}; \; v_i(0) = v_{i0};\; x_{i0}, v_{i0} \in \R^n; \; i = 1,\ldots, N.
\end{split}
\right. 
\end{equation}
Here, we use the shorthand notation $x_{ij} = x_i - x_j$, $v_{ij} = v_i - v_j$.  The system (CS) describes the evolution of $N$ agents with positions $x_i$, velocities $v_i$, and (fixed) masses $m_i$.  One obtains the system \eqref{e:m} from \eqref{e:CS} by first passing through a kinetic model (c.f. \cite{HT2008}, \cite{CFRT2010}, \cite{MuchaPeszek2018} and references therein).  One can easily (but formally) derive \eqref{e:m} from the kinetic equation by taking appropriate moments of the kinetic equation, then making the `monokinetic ansatz' $f(x,v,t) = \rho(x,t)\d_0(v-u(x,t))$, where $u$ is the macroscopic velocity.  Rigorous derivations are given in \cite{KMT2015}, \cite{FigalliKang2019}. 

The salient feature of the system \eqref{e:CS} is that the interaction of the agents, governed by the communication protocol $\phi$, tends to align the velocities.  One can make stronger statements about this phenomenon by making assumptions on, for example, the communication weight $\phi$, the coupling strength $\k$, and/or connectivity properties of the initial configuration of agents.  There is a wide literature dedicated to formulating appropriate assumptions for concluding that velocity alignment and `flocking' occur; here, `velocity alignment' means that the diameter of the velocities tends to zero, and `flocking' means that the diameter of the agents stays bounded for all time.  As one should expect, the hydrodynamic system \eqref{e:m} also exhibits a tendency for $\phi$ to align the velocities; contributors to velocity alignment and flocking at the discrete level also tend to produce hydrodynamic alignment and flocking.  

Since flocking is a phenomenon that by definition occurs on long time-scales, a prerequisite for flocking to occur is of course the global-in-time existence of a solution.  Let us now turn our attention to consideration of the question of wellposedness, which will be the focus of the present work.  

\subsection{The Roles of Vacuum and of the Problem Domain for Wellposedness}

When studying wellposedness of the hydrodynamic model, one of the first choices one must make is the domain on which to ask that the PDE should be satisfied.  One primary consideration in this choice is how one wants to deal with the vacuum and/or regions of low density.  Such regions create difficulty for the wellposedness theory, because positive density is what drives the alignment force on the right side of \eqref{e:m}$_2$ and thus steers the system away from Burgers'-type blowup.  As alternatives to the formulation \eqref{e:m} listed above, one can also pose the problem on all of $\R^n$ (by imposing a velocity in the vacuum), or in the periodic setting $\T^n$ (where one can avoid discussion of the vacuum entirely by propagating a lower bound on the density).  The periodic setting is useful for studying flocking `in the bulk', and for certain kernels, a wellposedness theory is currently not available without such a lower bound.  (For strongly singular kernels---which in the hydrodynamic setting means $\phi(x)\sim |x|^{-n-s}$, $s\in (0,2)$---the existing wellposedness theory relies on a lower bound on $\rho$ to keep the dissipation of energy active.  C.f. \cite{ShvydkoyTadmorI}, \cite{ShvydkoyTadmorII}, \cite{ShvydkoyTadmorIII}, \cite{DKRT}, \cite{KT}, \cite{Leslie2019} for the 1D theory, and  \cite{Shvydkoy2018NearlyAligned}, \cite{DanchinMuchaPeszekWroblewski2018} for results in higher dimensions.)  However, at the largest scales, the full space setting is a more appropriate choice, whether the PDE is to be satisfied on all of $\R^n$ or on $\O(t)$ only.  We refer to these two formulations on the full space using the phrases `with vacuum velocity' (abbreviated (VV)) and `without vacuum velocity' (abbreviated (NVV)).

If one works with (VV), one can avoid the problem of tracing the dynamics of $\O(t)$ itself; most of the wellposedness literature on the full space setting uses this framework.  However, this technical simplification comes at a cost.  Physically, it requires the vacuum region to move as if it contained massless agents influenced by the dynamics inside $\O(t)$.  And from a technical standpoint, the fact that one must choose an initial velocity inside the vacuum is problematic.  Ideally, such data should be chosen so as to be compatible with the existence or non-existence of a solution to the formulation without vacuum velocity.  However, it is not a priori obvious whether a compatible choice necessarily even exists, and proving finite-time blowup for a given choice of vacuum data does not imply blowup for (NVV).  (On the other hand, proving global-in-time existence for (VV) for a given choice of initial data does guarantee existence for, and thus compatibility with, (NVV).)

The primary purpose of this note is to address the possible discrepancy between (VV) and (NVV), in the context of the 1D Euler Alignment model with smooth $\phi$, where the formulation (VV) is essentially resolved at the level of strong solutions.  We treat (NVV) directly and characterize the initial data that lead to strong solutions.  We show that, in fact, every solution of (NVV) can be extended to a strong solution of (VV) in a natural way.  We also consider the case of `weakly singular' $\phi$, which for the 1D hydrodynamic model means that $\phi(x)\sim |x|^{-s}$ near the origin, for some $s\in (0,1)$.  Wellposedness of the model with these kernels has been studied in \cite{Tan2019WeakSing}.  The latter resolves the case of subcritical and supercritical initial data, and shows that some critical initial data leads to finite-time blowup.  However, the analysis of \cite{Tan2019WeakSing} does not apply to all critical initial data.  Our second main result shows that a large class of critical initial data leads to global-in-time existence and thus slightly sharpens the criteria in \cite{Tan2019WeakSing}.

This paper is organized as follows.  In Section \ref{s:Lag}, we write down the Lagrangian formulation for \eqref{e:m} and make a simple observation that dramatically simplifies the system in 1D.  We state our first result formally here.  In section \ref{s:pf}, we give a self-contained proof of this result.  In section \ref{s:CT}, we compare our Theorem with previous results, especially the `critical threshold' criterion of \cite{CCTT2016} for (VV). In Section \ref{s:WS}, we apply the simplification we observed for the smooth case to prove bounds on the separation of particle trajectories, valid for both smooth and weakly singular kernels.  Our result on the wellposedness for weakly singular kernels follows from these bounds.  

\section{The Lagrangian Formulation}
\label{s:Lag}

When \eqref{e:m} is recast in Lagrangian variables, its relationship to the discrete system is especially self-evident:
\begin{equation}
\label{e:Lag}
\left\{ 
\begin{split}
& \dot{X}_\a = V_\a,  \\
& \dot{V}_\a = \k \int_{\O} \phi(X_{\a\g})V_{\g\a}\dm_\g, \\
& X_\a(0) = \a, \;V_\a(0) = u_0(\a),
\end{split}
\quad \quad 
\a\in \O.\right. 
\end{equation}
Above, we used the notation $X_\a = X(\a,t)$ to denote the flow map, $V_\a=V(\a,t)$ to denote the velocity, and $X_{\a\g} = X_\a - X_\g$, $V_{\a\g} = V_\a - V_\g$.  Given an initial mass measure $m$ and a solution $(X_\a,V_\a)_{\a\in \O}$, the pushforward measure $m^t(E) = m(X^{-1}(E\cap \overline{\O}(t)),t))$ then satisfies \eqref{e:m}$_{1}$ in the sense of distributions:
\begin{equation}
\int_{\R^n} \xi(y,T) \dm^T_y - \int_{\R^n} \xi(y,0) \dm_y = \int_0^T \int_{\R^n} \frac{D\xi}{Dt}(y,s) \dm^t_y,
\quad \quad \xi\in C^\infty_c(\R^n).
\end{equation}
where $\frac{D}{Dt}$ is the material derivative.  Of course, the initial mass measure we have in mind is $\dm_\g = \rho_0(\g)\dd\g$, which allows us to define $\rho(x,t) = \frac{\dm^t_x}{\dx}$, the Radon-Nikodym derivative of $m$ with respect to Lebesgue measure; the latter gives $\rho(X(\a,t),t) = \rho_0(\a) (\det\n_\a X(\a,t))^{-1}$ if $\det \n_\a X>0$.  Since $m$ need not be absolutely continuous with respect to Lebesgue measure in order to make sense of \eqref{e:Lag}, one can sometimes apply results on \eqref{e:m} directly to the discrete model \eqref{e:CS} by using an initial mass measure of the form $m = \sum_{i=1}^N m_i \d_{x_i(0)}$, where $\d_a$ denotes the Dirac delta distribution based at $a\in \R^n$.  For simplicity, we only consider the case where $\rho_0$ is at least continuous on $\O$ (though it may `jump' across $\p\O$), but we continue to use the notation suggestive of a more general mass measure.  

\subsection{Local Existence for Smooth Kernels}
For the convenience of the reader (and because it can be done quickly), we now give a brief discussion of local existence for \eqref{e:m}, for smooth kernels.  We do not deal with the sharpest possible spaces, nor do we attempt to incorporate general locally integrable kernels.  We instead refer the reader to the following references for more details: \cite{CCTT2016}, \cite{HaKangKwon2015}, \cite{Tan2019WeakSing}, \cite{LearShvydkoy2019}.   

If $\phi$ is a smooth kernel, $\O$ is a given open subset of $\R^n$, and $u_0$ is sufficiently smooth on $\O$, then it is straightforward to show that a local-in-time classical solution to \eqref{e:Lag} exists for all time and is as regular as $u_0$ is.  (That is, if $u_0\in C^k(\O)$, then $(X,V)\in C^2([0,\infty);\dot{C}^k(\O))\times C^1([0,\infty);C^k(\O))$.) This solution to \eqref{e:Lag} gives rise to some notion of solution to \eqref{e:m} as long as $X(\cdot,t):\O\to \O(t)$ remains bijective.  We would like to have a notion of solution to \eqref{e:m} that makes sense pointwise, so we specify that $\dm_\g = \rho_0(\g)\dd\g$, with $\rho_0\in C^h(\O)$, $h\in \N\cup \{0\}$, $u_0\in C^k(\O)$, $k>h$, and that $\O$ satisfies the uniform exterior ball condition at every point of its boundary\footnote{There exists $r>0$ such that, for every $z\in \p\O$, there exists $x\notin \O$, with $|x-z| = r$, such that $B(x,r)\cap \O = \emptyset$.  If $n=1$, it is enough for $\O$ to be a finite union of separated open intervals: $\O = \bigcup_{j=0}^J (\ell_j, r_j)$, with $r_{j-1}<\ell_j$ for each $j\in \{1,\ldots, J\}$}. We then define 
\[
\rho(x,t) = \rho_0(X^{-1}(x,t))(\det\n_\a X(X^{-1}(x,t),t))^{-1}, \quad 
u(x,t) = V(X^{-1}(x,t),t),
\quad \quad x\in \O(t).
\]
It follows that for every $t>0$ and $x\in \O(t)$ that there exists a space-time neighborhood $U$ of $(x,t)$ with the following property: on the set $U$, we have that $\rho$ and $u$ are continuous in time, as well as $h$ and $k$ times differentiable in space, respectively, as long as $X(\cdot, t)$ remains bijective and $\det \n_\a X>0$. The latter conditions are true at least on some finite time interval under our assumptions.  However, note carefully that the second condition (positivity of $\det \n_\a X$) does \textit{not}, in general, imply the first (bijectivity of $X(\cdot, t)$).  We refer to solutions to \eqref{e:m} of the type discussed above as \textit{regular solutions}.


\subsection{Special Properties in 1 Dimension}
Suppose that $n=1$ and $\phi$ is at least locally  integrable.  Then the right side of \eqref{e:Lag}$_2$ is a perfect time derivative. Indeed, denoting 
\begin{equation}
\f(x) = \int_{0}^x \phi(y)\dy,
\end{equation}
we see that \eqref{e:Lag}$_2$ can be written as 
\begin{align*}
\dot{V}_\a & = -\k\frac{\dd}{\dt} \int_\O \f(X_{\a\g}) \dm_\g,
\end{align*}
or
\begin{equation}
\label{e:Valph}
V_\a = u_0(\a) +\k(\f * m)(\a) - \k\int_\O \f(X_{\a\g}) \dm_\g.
\end{equation}
Equation \eqref{e:Valph} has several important consequences.  First of all, it gives the following time-independent quantity:
\begin{equation}
\psi_0(\a):= u_0(\a) + \k(\f * m)(\a) = V_\a + \k\int_\O \f(X_{\a\g}) \dm_\g,
\quad \a\in \O.
\end{equation}
(We stress that $\psi_0$ is defined only in $\O$.)  We can thus write \eqref{e:Lag} as a first-order system:
\begin{equation}
\label{e:1Dchar}
\left\{ 
\begin{split}
& \dot{X}_\a = \psi_0(\a) - \k\int_\O \f(X_{\a\g})\dm_\g, \\
& X_\a(0) = \a, \quad \psi_0 = u_0 + \k\f*m, \quad \a\in \O.
\end{split}
\right. 
\end{equation}
And finally, one has the following equation for the evolution of the distance between two Lagrangian trajectories.  Denoting $\psi_{\b\a} = \psi_0(\b) - \psi_0(\a)$, we can write
\begin{equation}
\label{e:dif}
\dot{X}_{\b\a} = \psi_{\b\a} - \k \int_{X_\a}^{X_\b} \int_\O  \phi(y - X_\g)  \dm_\g \dy,
\quad \quad \a,\b\in \O, \;\a<\b.
\end{equation}
Equation \eqref{e:dif} gives crucial information about the wellposedness of \eqref{e:m}.  Indeed, establishing bijectivity of the flow map $X(\cdot, t)$ is one of the main steps in proving the global-in-time existence of solutions to \eqref{e:m}, so lower bounds on $X_{\b\a}$ like the ones we obtain below using \eqref{e:dif} are key.  

We can now state our main Theorem on the 1D formulation (NVV), with smooth $\phi$.  (See Section \ref{s:WS} for a discussion of what is possible for certain kernels with weakly singular kernels.)

\begin{THEOREM}
	\label{t:main}
	Suppose $\phi$ is a smooth, radially decreasing interaction kernel on $\R$.  Let $\O$ be a finite union of intervals: $\O = \bigcup_{j=0}^J (\ell_j, r_j)$, with $r_{j-1}<\ell_j$ for each $j\in \{1,\ldots, J\}$.  Assume $\rho_0\in C^{h}(\O)$ and $u_0\in C^{k}(\O)$ ($h,k\in \N\cup \{0\}$, $k>h$).  Then $\eqref{e:m}$ has a global-in-time regular solution $(\rho, u)$ associated to the initial data $(\rho_0, u_0)$, if and only if $\psi_0=u_0 + \f*\rho_0$ is monotonically increasing on $\O$. Furthermore, if $\psi_0$ is monotonically increasing on $\O$, then $u_0$ has an extension $\widetilde{u}_0$ to all of $\R$ such that $(\rho_0, \widetilde{u}_0)$ yields a global-in-time solution of the Euler Alignment model with vacuum velocity, which is as regular as the initial data $(\rho_0, \widetilde{u}_0)$ allows.  If $\psi_0(r_{j-1})<\psi_0(\ell_{j})$ for each $j$, then $\widetilde{u}_0$ can be taken to be as smooth as $u_0$.  
\end{THEOREM}

\section{Global-in-Time Existence without Vacuum Velocity}
\label{s:pf}

In this section, we restrict attention to the case $n=1$.  In light of the local existence theory for smooth kernels discussed in the previous section, we now treat the bijectivity of $X$ and the positivity of $\p_\a X$.  We frame some of the discussion in terms of more general locally integrable kernels so that it does not have to be re-done in Section \ref{s:WS} when we discuss weakly singular kernels.

The following Proposition is an easy consequence of equation \eqref{e:dif}. 
\begin{PROP}
	\label{p:psibasic}
	Assume $\phi\in L^1_\loc(\R)$ and $u_0$, $\rho_0$ are given and sufficiently smooth on the bounded open set $\O = \{\rho_0>0\}\subset \R$.  A necessary condition for the global-in-time existence of a classical solution to \eqref{e:m} is that the function $\psi_0 = u_0 + \f*\rho_0$ should be monotonically increasing on $\O$. On the other hand, a sufficient condition for the bijectivity of the flow map $X(\cdot, t):\O\to \O(t)$ (on the time interval of existence of \eqref{e:Lag}) is that $\psi_0$ is strictly increasing on $\O$.  
\end{PROP}

\begin{proof}
	If $\psi_0(\a)>\psi_0(\b)$, but $\a<\b$, then \eqref{e:dif} shows that $X_{\b\a}(t) \le (\b - \a) - \k (\psi_0(\a) - \psi_0(\b))$ on the time interval of existence, which is at most $\frac{\b - \a}{\k(\psi_0(\a)-\psi_0(\b))}$.
	
	Assume on the other hand that $\psi_0$ is strictly increasing, and pick $\a<\b\in \O$.  If $X_{\b\a} = 0$ at some first time $t_0\in (0,+\infty)$, then $\dot{X}_{\b\a}(t_0) = \psi_{\b\a}>0$, which indicates that $X_{\b\a}(t)<0$ at some time $t$ previous to $t_0$, a contradiction.
\end{proof}

Despite its simplicity, Proposition \ref{p:psibasic} is, to the best of the author's knowledge, new.  Observe that it does not require any knowledge of $\phi$ other than its local integrability.

Let us similarly look at $\p_\a X$, at first without specifying more about $\phi$.  Differentiating \eqref{e:1Dchar}$_1$ with respect to $\a$ yields 
\begin{equation}
\p_\a \dot{X}_\a = \psi_0'(\a) - \k \left( \int_\O \phi(X_{\a\g})\dm_\g\right) \p_\a X_\a. 
\end{equation}
Or integrating, 
\begin{equation}
\p_\a X_\a(t) = \exp\left( - \k \int_0^t \int_\O \phi(X_{\a\g})\dm_\g \ds \right) + \psi_0'(\a) \int_0^t \exp\left( - \k \int_s^t \int_\O \phi(X_{\a\g})\dm_\g \dr \right) \ds.  
\end{equation}
Clearly $\p_\a X_\a(t)$ reaches zero in finite time (prior to time $t = |\psi_0'(\a)|^{-1}$) if $\psi_0'(\a)<0$, in agreement with Proposition \ref{p:psibasic}.  So let us assume that $\psi_0'(\a)\ge 0$.  In order to get a lower bound on $\p_\a X_\a(t)$, we need one additional assumption.  Suppose that 
\begin{equation}
\label{e:convbd}
\k \int_\O \phi(X_{\a\g})\dm_\g \le C,
\quad \text{ for all } \a\in \O,
\end{equation}
on the time interval of existence. Then we get the lower bound 
\begin{equation}
\label{e:paXlwr}
\p_\a X_\a(\a,t) \ge \exp(- C t) + \frac{\psi_0'(\a)}{C}[1 - \exp(-Ct)].
\end{equation}
In the case where $\phi$ is bounded, \eqref{e:convbd} is trivially satisfied with $C = \k\|\phi\|_{L^\infty} M_0$ (where $M_0 = \|m\|$, the total mass, i.e., the total variation of $m$).  If $\phi$ has a singularity at zero, it is no longer clear whether \eqref{e:convbd} holds.  We will consider this case in Section \ref{s:WS}.  We record the estimate \eqref{e:paXlwr} as a Proposition:

\begin{PROP}
\label{p:paXlwr}
Assume $\phi\in L^1_\loc(\R)$ and $u_0$, $\rho_0$ are given and sufficiently smooth on the bounded open set $\O = \{\rho_0>0\}\subset \R$.  Assume that $\psi_0'(\a)\ge 0$ for all $\a\in \O$ and that \eqref{e:convbd} holds.  Then on the time interval of existence of the solution to \eqref{e:Lag}, the lower bound \eqref{e:paXlwr} holds for all $\a\in \O$.
\end{PROP}

Finally, let us use the assumption that $\phi$ is bounded in order to conclude bijectivity of $X$ when $\psi_0$ is just monotonically increasing rather than strictly increasing.  If $\phi$ is bounded, we can actually write down an explicit lower bound for $X_{\b\a}$ if $\psi_{\b\a}\ge 0$.  Indeed, from \eqref{e:dif}, we see that 
\[
\dot{X}_{\b\a} \ge \psi_{\b\a} - \k\|\phi\|_{L^\infty} M_0 X_{\b\a},
\]
and thus 
\begin{equation}
\label{e:XbalwrL}
X_{\b\a} \ge  (\b - \a) \exp( -\k\|\phi\|_{L^\infty} M_0 t) + \frac{\psi_{\b\a}}{\k \|\phi\|_{L^\infty} M_0}( 1 - \exp( -\k\|\phi\|_{L^\infty} M_0 t)).
\end{equation}

We have now proved that $\p_\a X(t)\ge c(t)>0$ and that $X$ remains bijective for all time, provided that $\psi_0$ is monotonically increasing.  This proves the existence part of Theorem \ref{t:main}.  The extension to a solution of (VV) is now a triviality: If $\psi_0:\O\to \R$ is increasing on $\O$, then it has an increasing extension $\widetilde{\psi}_0:\R\to \R$ which can be taken to be as smooth as $\psi_0$ if $\psi_0(r_{j-1})<\psi_0(\ell_{j})$ for each $j$, and can be taken to be piecewise as smooth as $\psi_0$ if $\psi_0(r_{j-1})=\psi_0(\ell_{j})$ for some $j$.  Then one can define $\widetilde{u}_0:=\psi_0 - \f*\rho_0$ to finish.  

\begin{REMARK}
We could have of course derived the bound \eqref{e:paXlwr} (with $C = \k \|\phi\|_{L^\infty}M_0$) from \eqref{e:XbalwrL} in the case of bounded kernels.  We  choose to work with \eqref{e:paXlwr} directly in order to include it more naturally later in our discussion of weakly singular kernels in Section \ref{s:WS}.
\end{REMARK}


\section{Further Discussion on Global Existence}
\label{s:CT}

To get local existence for the formulation (VV), one can, for example, add a viscous regularization to the model (local existence for the regularized model can be established easily). Then one proves a priori estimates independent of the viscosity parameter to get a standard continuation criterion for a solution that is known to exist on the time interval $[0,T)$:
\begin{equation}
\label{e:BKM}
\int_0^T |\n u(x,t)|\dt <+\infty.
\end{equation}
See for example \cite{LearShvydkoy2019}, \cite{Tan2019WeakSing}, \cite{CCTT2016} for more details.  (Of these, only \cite{LearShvydkoy2019} explicitly includes the viscous regularization argument.) 
The usual way of establishing a bound on $|\n u|_{L^\infty}$, and thus existence for (VV), involves the quantity $e = \diver u + \phi*\rho$, which satisfies the equation 
\begin{equation}
\label{e:econt}
\p_t e + \diver(u e) = (\diver u)^2 - \tr(\n u)^2.  
\end{equation}
In dimension $n=1$, the right hand side of this equation is zero, and one has 
\begin{equation}
\frac{\dd}{\dt} e(X(\a,t),t) = -[e(e - \k \phi*\rho)](X(\a,t),t).
\end{equation}
In the case of smooth kernels, one can read off from this equation that $e(X(\a,t),t)$ moves toward the (bounded) quantity $\k \phi*\rho$ if $e_0(\a)>0$, tends to $-\infty$ in finite time if $e_0(\a)<0$, and remains zero for all time if $e_0(\a) = 0$.  Again since $\phi*\rho$ is bounded, it follows that $\p_x u(X(\a,t),t)$ remains bounded for all $\a$ if and only if $e(X(\a,t),t)$ does, which occurs if and only if $e_0$ is everywhere nonnegative.  One therefore refers to the condition $e_0\ge 0$ as the \textit{critical threshold condition} for smooth kernels.  

Of course, the quantity $e_0 = \p_x u_0 + \phi*\rho_0$ is just the derivative of $\widetilde{\psi}_0$ from the previous section, and monotonically increasing $\widetilde{\psi}_0$ is the same as $e_0\ge 0$.  We emphasize, however, that $e_0\big|_\O \ge 0$ is not sufficient to guarantee existence of a solution to \eqref{e:m}.  This is clear from Theorem \ref{t:main}, and is illustrated even more vividly by the following example.  

\begin{EXAMPLE}
\label{ex:HKK}
	Let $\phi$ be any radially decreasing smooth interaction protocol.  Choose $\e>0$, then choose (if possible) $\d_0>0$ so that 
	\begin{equation}
	\k\int_{-\d_0}^{\d_0} \int_\R \phi(\b - \g) \;\dd\beta \; \dd\g = 2\e.
	\end{equation}
	If no such $\d_0$ exists, put $\d_0 = +\infty$. Choose $\d>0$ and define
	\[
	\O_+ = (\d, \tfrac12 + \d), 
	\quad 
	\O_- = (-\tfrac12 - \d, -\d),
	\quad  
	\O = \O_+ \cup \O_-,
	\] 
	\[
	\rho_0 = 1_\O,
	\quad u\big|_{\O_\pm} = \mp \e.  
	\]
	Then $\p_x u_0 \equiv 0$ on $\O$ and thus $\inf_\O e_0>0$.  However, we have global-in-time existence if $\d\ge \d_0$ and blowup at time prior to $T = \d/(\k\e)$ if $0<\d<\d_0$.  Indeed, one has $\psi_0(\d) - \psi_0(-\d) = -2\e + \k \int_{-\d}^{\d} \int_\O \phi(y - X_\g) \dm_\g \dy$.  If $\d<\d_0$, then this quantity is negative, so the characteristics originating in $\O_+$ and $\O_-$ cross in finite time.  On the other hand, if $\d\ge \d_0$, then $\psi_0$ is monotonically increasing on all of $\O$ and thus we have global-in-time existence.
\end{EXAMPLE}

The above example shows several points: First of all, it  demonstrates explicitly that $e_0\big|_\O \ge 0$ is not a sufficient condition for the existence of a solution to (NVV).  And although we used Theorem \ref{t:main} to find our sharp $\d_0$, one does not need to introduce the function $\psi_0$ in order to show that crossing of characteristics occurs for small enough $\d>0$.  In particular, one can construct similar explicit examples in higher dimensions (where the function $\psi_0$ is not available).  In higher dimensions, one can do this even if $\O$ is connected.  

Note also that the initial velocity in Example \ref{ex:HKK} is arbitrarily small.  This implies that any Theorem on existence of solutions to (NVV), even a small data result, must somehow take into account the geometry of $\O$.  The requirement that $\psi_0$ must be monotonically increasing reflects the separation between intervals (and thus the geometry of $\O$) in the following way: it guarantees that the alignment force `has enough room' to avoid collision of two adjacent intervals, by balancing the difference in velocities $u_0(\ell_j) - u_0(r_{j-1})$ with the difference $\phi*\rho_0(\ell_j) - \phi*\rho_0(r_{j-1})$ (the latter indicates the capacity of the trajectories $X(r_{j-1},t)$ and $X(\ell_j,t)$ to avoid collision).  The only other work known to the author that treats wellposedness of \eqref{e:m} directly (as opposed to the formulation (VV)) is \cite{HaKangKwon2015}, which claims a small data result in any dimension.  While their proof does establish a lower bound on $\p_\a X$, their smallness assumption does not account for the geometry of $\O$ and thus cannot be true in full generality, as demonstrated by Example \ref{ex:HKK}.  It would be interesting to study precisely what geometric condition is needed in order to rule out crossing of characteristics of the kind that is not detected by $\n_\a X$.  

In higher dimensions, the right side of \eqref{e:econt} is not zero, which complicates the analysis.  As a result, the wellposedness theory is far less developed in higher dimensions than in the 1D case.  Sufficient conditions for existence are available in \cite{TT2014}, \cite{HT2016}.  See also the recent work \cite{LearShvydkoy2019} of Lear and Shvydkoy for an analysis of the special case of `unidirectional' flows.

\begin{REMARK}
\label{r:extrest}
Another proof of the equivalence of global-in-time existence to (NVV) with monotonically increasing $\psi_0$ on $\O$ goes as follows: If $\psi_0$ is not monotonically increasing, then a global-in-time solution cannot exist, by Proposition \ref{p:psibasic}.  If $\psi_0$ is monotonically increasing, consider a monotonically increasing Lipschitz extension $\widetilde{\psi}_0$ to $\R$.  Then use the work of \cite{CCTT2016} to conclude global-in-time existence, since $e_0=\widetilde{\psi}_0'\ge 0$ everywhere (except possibly on $\p\O$).  The author feels, however, that the self-contained proof in Section \ref{s:pf} gives a more transparent picture of why the Theorem should be true.  
\end{REMARK}


\section{Weakly Singular Kernels}
\label{s:WS}

In this section, we discuss the implications of the simplified system \eqref{e:1Dchar} on the wellposedness theory for weakly singular kernels.  We give some bounds for the separation between trajectories, and we use these bounds to sharpen somewhat the known criteria sufficient for the existence of a global-in-time solution, in the case where the order of the singularity is less than $\frac12$.  In what follows, we set $\k=1$ and encode the coupling strength directly into the protocol $\phi$.

\subsection{Bounds on the Separation of Trajectories}
\label{ss:sep}

We first give a lower bound on the separation of particle trajectories, assuming an upper bound on $\phi$ of the form $\phi(x)\le \L|x|^{-s}$ near the origin.  Note that the case $s=0$ has already been discussed, so we don't include it in the statement of the following Proposition.

\begin{PROP}[A Lower Bound on the Separation of Trajectories]
	\label{p:Xlb}
	Assume that $\phi:\R\to \R_+$ is a radially decreasing, weakly singular interaction protocol, satisfying 
	\begin{equation}
	\label{e:phiup}
	\phi(x) \le \L|x|^{-s}, \quad |x|<R_0, \quad \text{ for some } s\in (0,1).
	\end{equation}
	Assume $\a,\b\in \O$, with $\a<\b$.  Then whenever $0<X_{\b\a} < 2R_0$, we have 
	\begin{equation}
	\label{e:Xdotlwr}
	\dot{X}_{\b\a} \ge \psi_{\b\a} - \frac{2^s \L M_0}{1-s} X_{\b\a}^{1-s}. 
	\end{equation}
	Consequently, 
	\begin{equation}
	\label{e:Xbalwr}
	\inf_{t\ge 0} X_{\b\a}(t) \ge \min\left\{ \b -\a, \;\; \left( \frac{(1-s)\psi_{\b\a}}{2^s \L M_0} \right)^{\frac{1}{1-s}}, \;\; 2R_0 \right\},
	\end{equation}
	and 
	\begin{equation}
	\label{e:liminfXba}
	\liminf_{t\to +\infty} X_{\b\a} \ge \min\left\{ \left( \frac{(1-s)\psi_{\b\a}}{2^s \L M_0} \right)^{\frac{1}{1-s}}, \;\;2R_0\;\;\right\};
	\end{equation}
	if $\b-\a$ is less than the value on the right side of \eqref{e:liminfXba}, then $X_{\b\a}$ increases monotonically at least until it reaches this value.  
\end{PROP}
\begin{proof}
	We prove only the validity of the differential inequality \eqref{e:Xdotlwr}; the other statements follow immediately.  Note that for fixed $\a,\b$, the quantity $\int_{X_\a}^{X_\b} \phi(y - z) \dy$ is maximized when $z = \frac{X_\a + X_\b}{2}$, simply by virtue of the fact that $\phi$ is radially decreasing.  We thus have 
	\[
	\int_{X_\a}^{X_\b} \phi(y - z)\dy \le 2 \int_0^{\frac12 X_{\b\a}} \phi(r)\dr \le \frac{2^s\L}{1-s} X_{\b\a}^{1-s},
	\quad \quad \text{ if } X_{\b\a} \le 2R_0. 
	\]
	Substituting this bound into \eqref{e:dif} (for each $\g\in \O$), we obtain \eqref{e:Xdotlwr} whenever $X_{\b\a}<2R_0$.
\end{proof}

One can prove an upper bound in a similar way, under an assumption of the form $\phi(x)\ge \l |x|^{-s}$ near the origin.  For this bound, we do include the case $s=0$ of bounded kernels. 

\begin{PROP}[An Upper Bound on the Separation of Trajectories]
	\label{p:Xub}
	Assume that $\phi:\R\to \R_+$ is a radially decreasing smooth or weakly singular interaction protocol, satisfying 
	\begin{equation}
	\label{e:philwr}
	\l|x|^{-s} \le \phi(x), \quad |x|<R_0, \quad \text{ for some } s\in [0,1).
	\end{equation}
	For all pairs $(\a,\b)$ satisfying 
	\begin{equation}
	\label{e:uphyp}
	\max\left\{ \b-\a, \left( \frac{(1-s)\psi_{\b\a}}{\l m([\a,\b])} \right)^{\frac{1}{1-s}} \right\} \le R_0, 
	\end{equation}
	one has 
	\begin{equation}
	\label{e:Xdotup}
	\dot{X}_{\b\a} 
	\le \psi_{\b\a} - \frac{\l m([\a,\b])}{1-s}X_{\b\a}^{1-s}.
	\end{equation}
	Consequently, 
	\begin{equation}
	\sup_{t\ge 0} X_{\b\a}(t) \le \max\left\{ \b-\a, \left( \frac{(1-s)\psi_{\b\a}}{\l m([\a,\b])} \right)^{\frac{1}{1-s}} \right\} \le R_0.
	\end{equation}
\end{PROP}
\begin{proof}
	Assume $X_{\b\a}\in (0,R_0)$.  For $z\in (X_\a, X_\b)$, the integral $\int_{X_\a}^{X_\b} \phi(y - z)\dy$ is minimized when $z = X_\a$ or $z=X_\b$, simply by virtue of the fact that $\phi$ is radially decreasing.  Therefore 
	\begin{align*}
	\int_\O \int_{X_\a}^{X_\b} \phi(y - X_\g)\dy\dm_\g
	& \ge \int_{\a}^\b \int_{X_\a}^{X_\b} \phi(y - X_\g)\dy\dm_\g \\
	& \ge m([\a,\b]) \int_0^{X_{\b\a}} \phi(r)\dr  
	\ge \frac{\l m([\a,\b])}{1-s}X_{\b\a}^{1-s}.
	\end{align*}
	Substituting the above into \eqref{e:dif} yields \eqref{e:Xdotup}.
\end{proof}

\begin{REMARK}
Both the Propositions above have applications for the wellposedness theory; see the discussion in the following subsection.  One might also be tempted to build a flocking Theorem out of Proposition \ref{p:Xub}.  This is certainly possible but does not yield a very strong statement, for reasons outlined presently.  The natural hypotheses for a flocking Theorem based on Proposition \ref{p:Xub} would be for $\O$ to be chain connected at scale $R_0/2$ and that some chain $(\a_i)_{i=0}^N$ in $\overline{\O}$ satisfies 
\[
\max\left\{ \a_i - \a_{i-1}, \left( \frac{(1-s)\psi_{\a_{i}\a_{i-1}}}{\l m([\a_{i-1},\a_i])} \right)^{\frac{1}{1-s}}\right\}\le \frac{R_0}{2}, \quad i=1,\ldots, N, 
\] 
with $\O\subset \bigcup_{i=1}^N [\a_{i-1}, \a_i]$; this gives the conclusion that  $\diam\O(t)$ remains uniformly bounded for all time (by $NR_0/2$), and that the velocities align exponentially fast to a constant (following an argument of Morales, Peszek, and Tadmor \cite{MPT2019}, but with simplifications due to the a priori knowledge of the upper bound in Proposition \ref{p:Xub}).  As happy a situation as this might seem on its face, the result is actually rather weak.  Indeed, translating the hypothesis \eqref{e:uphyp} into a condition on $u_0(\b) - u_0(\a)$, we see that it requires
\begin{equation}
\label{e:uphypu}
u_0(\b) - u_0(\a) \le \frac{\l R_0^{1-s}}{1-s} m([\a,\b]) - \int_\a^\b \int_\O \phi(\eta - \g) \dm_\g \dd\eta,
\end{equation}
the right hand side of which need not be positive.  So, this condition can require that the trajectories $X(\a,t)$ and $X(\b,t)$ are initially moving toward each other in order to conclude a uniform upper bound on their separation, whereas we would ideally like an upper bound for particles initially moving away from one another. Despite this shortcoming, Proposition \ref{p:Xub} still has the following useful interpretation: If the trajectories initially at $\a$ and $\b$ (with $0<\b-\a<R_0$) are initially moving towards each other quickly enough (according to \eqref{e:uphypu}), then they can never be separated beyond some threshold distance.  Note that \eqref{e:uphypu} depends on the initial velocity field \textit{only at } $\a$ and $\b$.  So, once this threshold \eqref{e:uphypu} is met, the trajectories cannot be influenced by the velocity field outside $[\a,\b]$ to separate, no matter what this velocity field is.  
\end{REMARK}

\subsection{Wellposedness}

The wellposedness of the Euler Alignment model with weakly singular kernels has been analyzed by Tan in \cite{Tan2019WeakSing}. Tan's results can be summarized as follows:
\begin{itemize}
	\item The case $\inf e_0<0$ yields blowup, by essentially the same argument as for smooth kernels. 
	\item If $\inf e_0>0$, then $\phi*\rho$ satisfies a nonlinear maximum principle, which one can use to prove a uniform-in-time bound on $\|\rho(t)\|_{L^\infty}$, thus on $\|\phi*\rho(t)\|_{L^\infty}$ and then $\|\p_x u(t)\|_{L^\infty}$.  (Note that the argument of \cite{CCTT2016} does not immediately apply here.)
	\item If $e_0\equiv 0$ on an interval $[\a,\b]$ with $m([\a,\b])>0$, then $X_{\b\a}\to 0$ in finite time, giving finite-time blowup of the density.  This distinguishes the wellposedness theory for smooth and weakly singular kernels.  (The case where $e_0=0$ on a set of $m$-measure zero is left unresolved.)  
\end{itemize}

We now compare Tan's results with what can be obtained using the function $\psi_0$.  The conclusion for the case $\inf e_0<0$ is apparent from Proposition \ref{p:psibasic}.  One can also adapt Tan's nonlinear maximum principle argument to get a lower bound on $\p_\a X$ under the assumption that $\inf_\O \psi_0'>0$.  However, we do not give the details, since this does not provide any new information.  Tan's analysis of the case $e_0\equiv 0$ on $[\a,\b]$ relies on the inequality \eqref{e:Xdotup} in the special case $\psi_{\b\a} = 0$; this was available even without using the function $\psi_0$ due to the fact that the set $\{e_0 \equiv 0\}$ is invariant in time on the interval of existence.  

We now discuss the case where $e_0\ge 0$, but $e_0 = 0$ at a finite number of points inside $\O$; this case is not treated by \cite{Tan2019WeakSing}.  From our Proposition \ref{p:psibasic}, we can already see that such points cannot prevent the flow map from being bijective, since $\psi_0$ will still be strictly increasing.  What is not clear is whether or not the density can become infinite (or equivalently, whether $\p_\a X$ can become zero) at or near such points.  This is where the lower bound \eqref{e:Xbalwr} becomes useful.  

Let us recall that we have proved in Section \ref{s:pf} the lower bound \eqref{e:paXlwr} on $\p_\a X$ under the assumption \eqref{e:convbd}.  For $\phi$ satisfying \eqref{e:philwr} in the restricted range $s\in (0,\frac12)$, we make an additional assumption on $\psi_0$ that is weaker than Tan's, but which will still allow us to conclude \eqref{e:convbd}.  

Let us first of all assume that our $\psi_0$ satisfies 
\begin{equation}
\label{e:wkTan}
\psi_{\b\a}\ge c(\b - \a)^\mu,
\quad \quad 0\le (\b - \a)<R_1, \;\a,\b\in \O,
\end{equation}
for some $R_1<R_0$ $(R_1\ll 1)$ and some $\mu>1$.  This guarantees that $\psi_0$ is strictly increasing but allows for the possibility that its derivative vanishes somewhere.  Clearly, the larger $\mu$ is, the less stringent a requirement \eqref{e:wkTan} is; we now determine the range of $\mu$'s that will allow us to establish the upper bound \eqref{e:convbd} and thus the global-in-time existence of a solution.  Assuming $\psi_0$ satisfies \eqref{e:wkTan}, the lower bound \eqref{e:Xbalwr} guarantees that for $|\b - \a|<R_1$, we have 
\begin{equation}
\inf_{t\ge 0} |X_{\b\a}(t)| \ge \left( \frac{(1-s)|\psi_{\b\a}|}{2^s \L M_0} \right)^{\frac{1}{1-s}}\ge \left( \frac{(1-s)c|\b - \a|^\mu}{2^s \L M_0} \right)^{\frac{1}{1-s}}.
\end{equation}
Denote 
\[
R_2 := \left( \frac{(1-s)cR_1^\mu}{2^s \L M_0} \right)^{\frac{1}{1-s}},
\quad \quad 
\widetilde{c} = \left( \frac{(1-s)c}{2^s \L M_0} \right)^{\frac{1}{1-s}}.
\]
Then 
\begin{align*}
\int_\O \phi(X_{\a\g})\dm_\g 
& = \int_{\O\cap \{|\a-\g|\ge R_1\}} \phi(X_{\a\g})\dm_\g + \int_{\O\cap \{|\a-\g|< R_1\}} \phi(X_{\a\g})\dm_\g \\
& \le \phi(R_2) M_0 + \L \widetilde{c}^{-s}\int_{\O\cap |\a - \g|<R_1} |\a - \g|^{-\frac{\mu s}{1-s}}\dm_\g, 
\end{align*}
which is uniformly bounded by a constant if $\rho_0$ is bounded (which we assume) and if $\frac{\mu s}{1-s}<1$.  Rearranging the latter yields 
\begin{equation}
1<\mu<\frac{1-s}{s},
\end{equation}
which is a nonempty range when $s\in (0,\frac12)$.  

We can now give the following Theorem.  The hypotheses are adapted so that we can use the local existence theory for weakly singular kernels in \cite{Tan2019WeakSing}.
\begin{THEOREM}
Assume $\phi$ is a radially decreasing, weakly singular kernel satisfying \eqref{e:phiup} for some $s\in (0,1)$.  Assume that $\O$ is a finite union of disjoint open intervals: $\O = \bigcup_{j=0}^N (\ell_j, r_j)$, with $r_{j-1}<\ell_j$ for each $j\in \{1,\ldots, J\}$.  Assume $\rho_0\in H^k(\R)$ and $\psi_0\in H^{k+1}(\O)$, $k\ge 1$, with $\psi_0=u_0 + \f*\rho_0$ strictly increasing on $\overline{\O}$. Then there exists an associated global-in-time solution $(\rho,u)$ to \eqref{e:m} under either of the following assumptions:
\begin{enumerate}
\item $\inf_{\O} \psi_0'>0$, or 
\item $s\in (0,\frac12)$, and $\psi_0$ satisfies \eqref{e:wkTan} for some $\mu \in (1,\frac{1-s}{s})$.
\end{enumerate}
\end{THEOREM}

\begin{proof}
We use the extension/restriction argument suggested by Remark \ref{r:extrest}, with some additional verifications under the second assumption.  As a byproduct of the proof, we also obtain a solution to the formulation (VV) associated to any extension of $\psi_0$ to all of $\R$ that satisfies the requirements outlined in the proof.

Under the first assumption, $\psi_0$ has an extension $\widetilde{\psi}_0$ satisfying $\widetilde{\psi}_0'\in H^{k}(\R)$ and $\inf_{\R} \psi_0'>0$.  Thus the result of Tan applies, and no further argument is needed. If $\psi_0$ satisfies the second assumption, we can again extend it to a strictly increasing $\widetilde{\psi}_0$ on all of $\R$, with $\widetilde{\psi}_0'\in H^k(\R)$.  However, Tan's result does not apply, since we may have $\widetilde{\psi}_0'=0$ somewhere; therefore we argue as follows.

We have already observed that the second assumption guarantees a uniform-in-time upper bound on  the quantity 
\[
\sup_{\a\in \O} \int_\O \phi(X_{\a\g})\dm_\g = \sup_{\a\in \O} \phi*\rho(X(\a,t),t).
\]
The latter implies a uniform upper bound on $\phi*\rho(y,t) = \int_{\O} \phi(y - X_\g)\dm_\g$ on all of $\R$.  Indeed, for $y\notin \O(t)$, we can use the trivial estimates 
\[
\int_\O \phi(y - X_\g)\dm_\g \le \int_\O \sum_{j=0}^N \phi(X_{\ell_i \g}) + \phi(X_{r_i \g}) \dm_\g 
\le 2(N+1)C
\]
if $\dist{y}{\O(t)}$ is small (less than $R_1$, say), and 
\[
\int_\O \phi(y - X_\g)\dm_\g \le \phi(R_1) M_0,
\]
otherwise.  So, we have a bound on $\phi*\rho$ that is uniform in time and space.  We can thus use the original argument of \cite{CCTT2016} to get a uniform bound on $\|e(t)\|_{L^\infty(\R)}$ and thus $\|\n u\|_{L^\infty(\R)}$.
\end{proof}

\begin{REMARK}
In assumption (2), we need $\psi_0$ to be increasing on $\overline{\O}$ rather than just on $\O$ only so that it is possible to obtain an extension $\widetilde{\psi}_0$ with the same smoothness as $\psi_0$.  The assumption \eqref{e:wkTan} is enough to guarantee that $X_{\ell_i r_{i-1}}$ remains strictly positive for all time even if $\psi_{\ell_i r_{i-1}} = 0$.  To see this, one can adapt the argument of Proposition \ref{p:Xlb}, taking into account that $\int_\O \phi(y-X_\g)\dm_\g$ is uniformly bounded in time and space if \eqref{e:wkTan} holds.  
\end{REMARK}

In closing, we remark that the value $s=\frac12$ plays a distinguished role in the theory for weakly singular kernels at the discrete and the kinetic levels (in any dimension), at least in terms of what can be proven about the equations.  See the works \cite{Pes2014}, \cite{Pes2015}, \cite{MuchaPeszek2018}, \cite{Minakowski2019}.  Future research may reveal that this role is a technical rather than a fundamental one; nevertheless, it would be interesting to understand whether there is a connection between the reasons this barrier appears at the particle/kinetic and at the hydrodynamic levels.  The two are not obviously related by the analysis above; in works on the discrete and kinetic models, the significance of $s=\frac12$ lies in the fact that $\phi$ is square integrable for $s\in (0,\frac12)$, whereas this fact does not appear to play a role in the analysis above.


\begin{thebibliography}{10}
	
	\bibitem{CFRT2010}
	J.~A. Carrillo, M.~Fornasier, J.~Rosado, and G.~Toscani.
	\newblock Asymptotic flocking dynamics for the kinetic {C}ucker-{S}male model.
	\newblock {\em SIAM J. Math. Anal.}, 42(1):218--236, 2010.
	
	\bibitem{CCTT2016}
	Jos\'e~A. Carrillo, Young-Pil Choi, Eitan Tadmor, and Changhui Tan.
	\newblock Critical thresholds in 1{D} {E}uler equations with non-local forces.
	\newblock {\em Math. Models Methods Appl. Sci.}, 26(1):185--206, 2016.
	
	\bibitem{CS2007a}
	Felipe Cucker and Steve Smale.
	\newblock Emergent behavior in flocks.
	\newblock {\em IEEE Trans. Automat. Control}, 52(5):852--862, 2007.
	
	\bibitem{CS2007b}
	Felipe Cucker and Steve Smale.
	\newblock On the mathematics of emergence.
	\newblock {\em Jpn. J. Math.}, 2(1):197--227, 2007.
	
	\bibitem{DanchinMuchaPeszekWroblewski2018}
	Rapha\"{e}l Danchin, Piotr~B. Mucha, Jan Peszek, and Bartosz Wr\'{o}blewski.
	\newblock Regular solutions to the fractional {E}uler alignment system in the
	{B}esov spaces framework.
	\newblock {\em Math. Models Methods Appl. Sci.}, 29(1):89--119, 2019.
	
	\bibitem{DKRT}
	Tam Do, Alexander Kiselev, Lenya Ryzhik, and Changhui Tan.
	\newblock Global {R}egularity for the {F}ractional {E}uler {A}lignment
	{S}ystem.
	\newblock {\em Arch. Ration. Mech. Anal.}, 228(1):1--37, 2018.
	
	\bibitem{FigalliKang2019}
	Alessio Figalli and Moon-Jin Kang.
	\newblock A rigorous derivation from the kinetic {C}ucker-{S}male model to the
	pressureless {E}uler system with nonlocal alignment.
	\newblock {\em Anal. PDE}, 12(3):843--866, 2019.
	
	\bibitem{HaKangKwon2015}
	Seung-Yeal Ha, Moon-Jin Kang, and Bongsuk Kwon.
	\newblock Emergent dynamics for the hydrodynamic {C}ucker-{S}male system in a
	moving domain.
	\newblock {\em SIAM J. Math. Anal.}, 47(5):3813--3831, 2015.
	
	\bibitem{HT2008}
	Seung-Yeal Ha and Eitan Tadmor.
	\newblock From particle to kinetic and hydrodynamic descriptions of flocking.
	\newblock {\em Kinet. Relat. Models}, 1(3):415--435, 2008.
	
	\bibitem{HT2016}
	Siming He and Eitan Tadmor.
	\newblock Global regularity of two-dimensional flocking hydrodynamics.
	\newblock {\em C. R. Math. Acad. Sci. Paris}, 355(7):795--805, 2017.
	
	\bibitem{KMT2015}
	Trygve~K. Karper, Antoine Mellet, and Konstantina Trivisa.
	\newblock Hydrodynamic limit of the kinetic {C}ucker-{S}male flocking model.
	\newblock {\em Math. Models Methods Appl. Sci.}, 25(1):131--163, 2015.
	
	\bibitem{KT}
	Alexander Kiselev and Changhui Tan.
	\newblock Global regularity for 1{D} {E}ulerian dynamics with singular
	interaction forces.
	\newblock {\em SIAM J. Math. Anal.}, 50(6):6208--6229, 2018.
	
	\bibitem{LearShvydkoy2019}
	Daniel Lear and Roman Shvydkoy.
	\newblock Existence and stability of unidirectional flocks in hydrodynamic
	euler alignment systems, 2019.
	
	\bibitem{Leslie2019}
	Trevor~M. Leslie.
	\newblock Weak and strong solutions to the forced fractional {E}uler alignment
	system.
	\newblock {\em Nonlinearity}, 32(1):46--87, 2019.
	
	\bibitem{Minakowski2019}
	Piotr Minakowski, Piotr~B. Mucha, Jan Peszek, and Ewelina Zatorska.
	\newblock {\em Singular Cucker--Smale Dynamics}, pages 201--243.
	\newblock Springer International Publishing, Cham, 2019.
	
	\bibitem{MPT2019}
	Javier Morales, Jan Peszek, and Eitan Tadmor.
	\newblock Flocking {W}ith {S}hort-{R}ange {I}nteractions.
	\newblock {\em J. Stat. Phys.}, 176(2):382--397, 2019.
	
	\bibitem{MuchaPeszek2018}
	Piotr~B. Mucha and Jan Peszek.
	\newblock The {C}ucker-{S}male equation: singular communication weight,
	measure-valued solutions and weak-atomic uniqueness.
	\newblock {\em Arch. Ration. Mech. Anal.}, 227(1):273--308, 2018.
	
	\bibitem{Pes2014}
	Jan Peszek.
	\newblock Existence of piecewise weak solutions of a discrete
	{C}ucker-{S}male's flocking model with a singular communication weight.
	\newblock {\em J. Differential Equations}, 257(8):2900--2925, 2014.
	
	\bibitem{Pes2015}
	Jan Peszek.
	\newblock Discrete {C}ucker-{S}male flocking model with a weakly singular
	weight.
	\newblock {\em SIAM J. Math. Anal.}, 47(5):3671--3686, 2015.
	
	\bibitem{Shvydkoy2018NearlyAligned}
	Roman Shvydkoy.
	\newblock Global existence and stability of nearly aligned flocks.
	\newblock {\em Journal of Dynamics and Differential Equations}, Aug 2018.
	
	\bibitem{ShvydkoyTadmorI}
	Roman Shvydkoy and Eitan Tadmor.
	\newblock Eulerian dynamics with a commutator forcing.
	\newblock {\em Transactions of Mathematics and Its Applications}, 1(1), 2017.
	
	\bibitem{ShvydkoyTadmorII}
	Roman {Shvydkoy} and Eitan {Tadmor}.
	\newblock Eulerian dynamics with a commutator forcing {II}: {F}locking.
	\newblock {\em Discrete Contin. Dyn. Syst.}, 37(11):5503--5520, 2017.
	
	\bibitem{ShvydkoyTadmorIII}
	Roman Shvydkoy and Eitan Tadmor.
	\newblock Eulerian dynamics with a commutator forcing {III}. {F}ractional
	diffusion of order {$0<\alpha<1$}.
	\newblock {\em Phys. D}, 376/377:131--137, 2018.
	
	\bibitem{TT2014}
	Eitan Tadmor and Changhui Tan.
	\newblock Critical thresholds in flocking hydrodynamics with non-local
	alignment.
	\newblock {\em Philos. Trans. R. Soc. Lond. Ser. A Math. Phys. Eng. Sci.},
	372:20130401, 2014.
	
	\bibitem{Tan2019WeakSing}
	Changhui {Tan}.
	\newblock {On the Euler-Alignment system with weakly singular communication
		weights}.
	\newblock {\em arXiv e-prints}, page arXiv:1901.02582, Jan 2019.
	
\end{thebibliography}

\def\cprime{$'$}

\end{document}